\documentclass[11pt,reqno]{amsart}
\usepackage{graphicx,amsmath,amsthm,verbatim,amssymb,enumerate}
\usepackage{xcolor}
\definecolor{darkgreen}{rgb}{0,0.75,0}
\definecolor{darkred}{rgb}{0.75,0,0}
\definecolor{darkmagenta}{rgb}{0.5,0,0.5}
\usepackage[colorlinks,citecolor=darkgreen,linkcolor=darkred,urlcolor=darkmagenta]{hyperref}

\newcommand{\mr}[1]{{\tt \href{http://www.ams.org/mathscinet-getitem?mr=#1}{MR#1}}}
\newcommand{\arxiv}[1]{{\tt \href{http://arxiv.org/abs/#1}{arXiv:#1}}}

\newcommand{\old}[1]{}

\newcommand{\abs}[1]{{\left\vert\kern-0.25ex #1
    \kern-0.25ex\right\vert}}

\providecommand{\flr}[1]{\left \lfloor #1 \right \rfloor } 
\providecommand{\cil}[1]{\left \lceil #1 \right \rceil } 
\DeclareRobustCommand{\SkipTocEntry}[5]{}

\newtheorem{theorem}{Theorem}[section]
\newtheorem{proposition}[theorem]{Proposition}

\newtheorem{lemma}[theorem]{Lemma}

\newtheorem{corollary}[theorem]{Corollary}

\theoremstyle{remark}
\newtheorem*{remark}{Remark}
\newtheorem*{example}{Example}

\numberwithin{counter}{section}

\theoremstyle{definition}
\newtheorem{definition}[theorem]{Definition}

\def\00{\mathbf{0}}

\def\E{\mathcal{E}}

\def\N{\mathbb{N}}

\def\R{\mathbb{R}}
\def\EE{\mathbb{E}}
\def\PP{\mathbb{P}}

\newcommand\norm[1]{\left\lVert#1\right\rVert} 
\def \restr{\big|} 

\begin{document}
\title{ Anomalous threshold behavior of long range random walks.}
\author[M. Murugan]{Mathav Murugan} 
\address{Department of Mathematics, University of British Columbia and Pacific Institute for the Mathematical Sciences, Vancouver, BC V6T 1Z2, Canada.}
\email[M. Murugan]{mathav@math.ubc.ca}

\author[L. Saloff-Coste]{Laurent Saloff-Coste$^\dagger$  }
\thanks{$\dagger$Both the authors were partially supported by NSF grant DMS 1404435}
\address{Department of Mathematics, Cornell University, Ithaca, NY 14853, USA}
\email[L. Saloff-Coste]{lsc@math.cornell.edu}
\date{\today}
\begin{abstract}
We consider weighted graphs satisfying sub-Gaussian estimate for the natural random walk.  On such graphs, we study symmetric Markov chains with heavy tailed jumps.
We establish a threshold behavior of such Markov chains when the index governing the tail heaviness (or jump index) equals the escape time exponent (or walk dimension) of the sub-Gaussian estimate.
In a certain sense, this generalizes the classical  threshold corresponding to the second moment condition.
\end{abstract}
\maketitle
\section{Introduction}
This work concerns a new threshold behavior of random walks on graphs driven by low moment measures.
As the title suggests, this work combines two lines of research that have been actively pursued: anomalous random walks and long range random walks.
The graphs were are interested in have a nearest neighbor random walk that satisfies sub-Gaussian estimates.  Sub-Gaussian estimates for nearest neighbor random walks are typical of many regular fractals
like Sierpinski gaskets, carpets and the Viscek graphs. See \cite{Kum} for a recent survey on such \emph{anomalous random walks}.
Another line of work that has received much attention recently is the long term behavior of random walks with \emph{heavy tailed jumps}. For example \cite{BL}, \cite{CK1}, \cite{CK2}, \cite{BBK}, \cite{BGK} are just a
few works in this direction.
In much of the existing literature the `jump index' $\beta$ is assumed to be in $(0,2)$. Our work is a modest attempt to understand the behavior of such random walks when $\beta \in (0,\infty)$.

The motivation for our work comes from a recent work by  the second author and Zheng \cite{SZ}. In \cite{SZ},
the behavior of long range random walks on groups was investigated for the full range of the jump index $\beta \in (0,\infty)$.
For random walks on groups there is a threshold behavior at $\beta=2$.
For graphs satisfying a sub-Gaussian heat kernel estimate, we show that the threshold behavior happens when the jump index $\beta$ equals the escape time exponent.

Let $\Gamma$ be an infinite, connected, locally finite graph endowed with a \emph{weight} $\mu_{xy}$. The elements of the set $\Gamma$ are called vertices. Some of the vertices are connected by an edge, in which case
we say that they are neighbors. The weight is a symmetric non-negative function on $\Gamma \times \Gamma$ such that $\mu_{xy}>0$ if and only if $x$ and $y$ are neighbors (in which case we write $x \sim y$).
We call the pair $(\Gamma,\mu)$ a \emph{weighted graph}.
The weight $\mu_{xy}$ on the edges induces a weight $\mu(x)$ on the vertices and a measure $\mu$ on subsets $A \subset \Gamma$ defined by
\[
 \mu(x):= \sum_{y: y \sim x} \mu_{xy}\hspace{5mm} \operatorname{and}\hspace{5mm} \mu(A):= \sum_{x \in A} \mu(x).
\]
Let $d(x,y)$ be the graph distance between points $x,y \in \Gamma$, that is the minimal number of edges in any edge path connecting $x$ and $y$. Denote the metric balls and their measures as follows
\[
 B(x,r):= \{ y \in \Gamma : d(x,y) \le r \}\hspace{5mm} \operatorname{and}\hspace{5mm} V_\mu(x,r):= \mu(B(x,r))
\]
for all $x \in \Gamma$ and $r \ge 0$.
 We assume that the measure $\mu$ is comparable to the counting measure in the sense that there exists $C_\mu \in [1,\infty)$ such that $\mu_x = \mu\left( \{x\}\right)$ satisfies
\begin{equation} \label{e-count}
 C_\mu^{-1} \le \mu_x \le C_\mu
\end{equation}
We consider weighted graphs $(\Gamma,\mu)$ satisfying the following uniform volume doubling assumption:
there exists $V :[0,\infty) \to (0,\infty)$, a strictly increasing continuous function and constants $C_D,C_h  > 1$ such that
\begin{equation} \label{e-vd}
 V(2r) \le C_D V(r)
\end{equation}
for all $r >0$ and
\begin{equation} \label{e-hom}
C_h^{-1} V(r) \le V_\mu(x,r) \le C_h V(r)
\end{equation}
for all $x \in \Gamma$ and for all $r >0$. It can be easily seen from \eqref{e-vd} that
\begin{equation} \label{e-vc}
 \frac{V(R)}{V(r) } \le C_D \left( \frac{R}{r} \right)^\alpha
\end{equation}
for all  $0 < r \le R$ and for all $\alpha \ge \log_2 C_D$. For the rest of the work, we will assume that
our weighted graph $(\Gamma,\mu)$ satisfies \eqref{e-count}, \eqref{e-vd} and \eqref{e-hom}.
\begin{remark}
 If $(\Gamma,\mu)$ satisfies \eqref{e-vd} and \eqref{e-hom}, we may assume that $V(n) = V_\mu(x_0,n)$ for some fixed $x_0$ and for all natural numbers $n$. For non-integer values we can extend it by linear interpolation.
 Since the graph is connected, infinite and locally finite, the function $V$ defined above
 is continuous, strictly increasing on $[0,\infty)$.
\end{remark}

There is a natural random walk $X_n$ on $(\Gamma,\mu)$ associated with the edge weights $\mu_{xy}$. The Markov chain is defined by the following
one-step transition probability
\[
  P(x,y)= \PP^{x}(X_1=y)= \frac{\mu_{xy}}{\mu(x)}.
\]
We will assume that there exists $p_0>0$ such that
\begin{equation}P(x,y) \ge p_0 \label{e-p0}\end{equation}
for all $x,y$ such that $x \sim y$.
We also consider $P$ as a Markov operator which acts on functions of $\Gamma$ by
\[
 Pf(x)= \sum_{y \in \Gamma} P(x,y)f(y).
\]

We will denote non-negative integers by $\N= \{ 0,1,2,\ldots \}$ and positive integers by $\N^*=  \{ 1,2,3,\ldots \}$.
For any non-negative integer $n$, the $n$-step transition probability $P_n$ is defined by $P_n(x,y)= \PP(X_n=y \mid X_0=x)= \PP^{x}(X_n=y)$. 
Define the \emph{heat kernel} of weighted graph $(\Gamma,\mu)$ by
\[
 p_n(x,y) := \frac{P_n(x,y)}{\mu(y)}.
\]
This Markov chain is symmetric with respect to the measure $\mu$, that is $p_n(x,y)=p_n(y,x)$ for all $x,y \in \Gamma$ and for all $n \in \N$.
We assume that there exists $\gamma > 1$ such that the following \emph{sub-Gaussian} estimates are true for the heat kernel $p_n$. There exist constants $c,C>0$ such that, for all $x,y \in \Gamma$
\begin{equation}\label{e-uep}
 p_n(x,y) \le \frac{C}{V(n^{1/\gamma})} \exp \left[ - \left( \frac{d(x,y)^\gamma}{Cn} \right)^{\frac{1}{\gamma-1}} \right], \forall n \ge 1
\end{equation}
and
\begin{equation}\label{e-lep}
(p_n+p_{n+1})(x,y) \ge \frac{c}{V(n^{1/\gamma})}  \exp \left[ - \left( \frac{d(x,y)^\gamma}{cn} \right)^{\frac{1}{\gamma-1}} \right], \forall n \ge 1 \vee d(x,y).
\end{equation}
Let  $\langle \cdot, \cdot \rangle$ denote the inner product in $\ell^2(\Gamma,\mu)$.
For the Markov operator $P$, define the corresponding Dirichlet form $\E_P$  by
\[
 \E_P(f):= \langle (I-P)f,f \rangle = \frac{1}{2}\sum_{x,y \in \Gamma} (f(x)-f(y))^2 \mu_{xy}
\]
for all $f \in \ell^2(\Gamma,\mu)$.
For any two sets $A,B \subset \Gamma$, the resistance $R_P(A,B)$ is defined by
\[
 R_P (A,B)^{-1} = \inf \left\{ \E_P(f,f):f \in \R^\Gamma, f\restr_A\equiv 1, f\restr_B\equiv 0 \right\}
\]
where $\inf \emptyset = +\infty$. By  \cite[Theorem 3.1]{GT}, we have the following estimate for the resistance. There exist constants $C_R, A>1$ such that
\begin{equation} \label{e-res}
C_R^{-1} \frac{r^\gamma}{V(r)}\le R_P(B(x,r),B(x,A r)^c) \le C_R \frac{r^\gamma}{V(r)}
\end{equation}
for all $x \in \Gamma$ and for all $r\ge 1$. Other related work that characterizes the sub-Gaussian estimates \eqref{e-uep} and \eqref{e-lep} are \cite{GT0} and \cite{BCK}.

The parameter $\gamma$ in \eqref{e-uep} and \eqref{e-lep} is sometimes called the `\emph{escape time exponent}' or `\emph{anomalous diffusion exponent}' or `\emph{walk dimension}'.
It is known that $\gamma \ge 2$ necessarily (see for instance \cite[Theorem 4.6]{CG}).
For any $\alpha \in [1,\infty)$ and for any $\gamma \in [2,\alpha+1]$,
Barlow constructs graphs of polynomial volume growth satisfying $V(x,r) \simeq (1+r)^\alpha$ and sub-Gaussian estimates \eqref{e-uep} and \eqref{e-lep} (see  \cite[Theorem 2]{Bar} and  \cite[Theorem 3.1]{GT}).
Moreover, these are the complete range of $\alpha$ and $\gamma$ for which sub-Gaussian estimates with escape rate exponent $\gamma$ could possibly hold for graphs of polynomial growth with growth exponent $\alpha$.

Let  $\phi:[0,\infty) \to [1,\infty)$ be a  continuous, regularly varying  function of positive index.
We say a Markov operator $K$ satisfies $(J_\phi)$ if it has symmetric kernel $k$ with respect to the measure $\mu$ and if there exists a constant $C_\phi>0$ such that
\begin{equation} \label{e-ker}
C_\phi^{-1} \frac{1}{V(d(x,y)) \phi(d(x,y))} \le k(x,y)= k(y,x) \le  C_\phi \frac{1}{V(d(x,y)) \phi(d(x,y))} \tag*{($J_\phi$)}
\end{equation}
for all $x,y \in \Gamma$. Let $k_n(x,y)$ denote the kernel of the iterated power $K^n$ with respect to the measure $\mu$.
If $K$ satisfies \ref{e-ker} and if $\phi$ is regularly varying with index $\beta >0$, then we say that $\beta$ is the \emph{jump index} of the random walk driven by $K$.
Here by random walk driven by $K$, we mean the discrete time Markov chain $(Y_n)_{n \in N}$ with transition probabilities given by
\[
 \PP(Y_1 = y | Y_0=x) = K \mathbb{1}_y(x)= k(x,y) \mu(y).
\]

\
We demonstrate threshold behavior as the jump index $\beta$ varies by analyzing the function
\begin{equation} \label{e-defpsi}
  \psi_K(n)=\norm{K^{2n}}_{1 \to \infty}= \norm{K^n}_{1\to 2}^2=\sup_{x \in \Gamma} k_{2n}(x,x)=\sup_{x,y \in \Gamma} k_{2n}(x,y)
\end{equation}
as $n \to \infty$ (see \cite{PSu} for a proof of \eqref{e-defpsi}).
The following theorem gives bounds on $\psi_K(n)$ that are sharp up to constants.
\begin{theorem}\label{t-main}
 Let $(\Gamma,\mu)$ be a weighted graph satisfying \eqref{e-count}, \eqref{e-vd}, \eqref{e-hom}, \eqref{e-p0} and suppose that
 its heat kernel $p_n$ satisfies the sub-Gaussian bounds \eqref{e-uep} and \eqref{e-lep} with escape time exponent $\gamma$. Let $K$ be a Markov operator symmetric with respect to
 the measure $\mu$ satisfying \ref{e-ker}, where $\phi:[0,\infty) \to [1,\infty)$ is a continuous regularly varying function of positive index. Then there exists a constant $C>0$ such that
 \begin{equation} \label{e-main}
  \frac{C^{-1}}{V(\zeta(n))} \le \psi_K(n) \le \frac{C}{V(\zeta(n))}
 \end{equation}
for all $n \in \N$, where $\zeta:[0,\infty) \to [1,\infty)$ is a continuous non-decreasing function which is an asymptotic inverse of
$t \mapsto {t^\gamma}/{ \int_{0}^t \frac{s^{\gamma-1}\,ds}{\phi(s)}}$.
\end{theorem}

\begin{example}
 We write $\phi$ in Theorem \ref{t-main} as $\phi(t)= \left((1+t)l(t)\right)^\beta $ where $l$ is a slowly varying function
 (we refer the reader to \cite[Chap. I]{BGT} for a textbook introduction on slowly and regularly varying functions).
 The function $\zeta$ of Theorem \ref{t-main} can be described more explicitly as follows:
 \begin{itemize}
  \item If $\beta > \gamma$, $\zeta(t) \simeq t^{1/\gamma}$.
  \item If $\beta < \gamma$, we have ${t^\gamma}/{ \int_{0}^t \frac{s^{\gamma-1}\,ds}{\phi(s)}} \simeq \phi(t) $ and $\zeta$ is essentially the asymptotic inverse of $\phi$, namely
  \[
   \zeta(t) \simeq t^{1/\beta} l_\# (t^{1/\beta})
  \]
 where $l_\#$ is the de Bruijn conjugate of $l$. For instance, if $l$ has the property that $l(t^a) \simeq l(t)$ for all $a>0$, then $l_\# \simeq 1/l$.
 \item If $\beta=\gamma$, the situation is more subtle. The function  $\eta(t)={t^\gamma}/{ \int_{0}^t \frac{s^{\gamma-1}\,ds}{\phi(s)}}$ is regularly varying of index $\gamma$ and $\eta(t) \le C_1 \phi(t)$ for some constant $C_1$.
 For example if $l \equiv 1$, we have $\eta(t) \simeq t^\gamma / \log t$ and $\zeta (t) \simeq (t \log t)^{1/\gamma}$. When $l(t) \simeq (\log t)^{\rho/\gamma}$ with $\rho \in \R$, then
 \begin{itemize}
  \item If $\rho >1$, $\eta(t) \simeq t^\gamma$ and $\zeta(t) \simeq t^{1/\gamma}$.
  \item If $\rho=1$,  $\eta(t) \simeq t^\gamma/ \log \log t$ and $\zeta(t) \simeq (t \log \log t)^{1/\gamma}$.
  \item If $\rho <1$, $\eta(t) \simeq t^\gamma/ (\log t)^{1-\rho}$ and $\zeta(t) \simeq \left(t (\log  t)^{1-\rho} \right)^{1/\gamma}$.
 \end{itemize}
 \end{itemize}

\end{example}

\begin{remark}\label{r-main}\leavevmode
\begin{enumerate}[(a)]
\item The conclusion of Theorem \ref{t-main} holds if $K$ is symmetric with respect to a different measure $\nu$ that is comparable to the counting measure in the sense described by \eqref{e-count}.
This can be seen by comparing $\psi_K$ with $\psi_{Q_\phi}$ where $Q_\phi$ will be defined in \eqref{e-qp}.
We simply use the definition \eqref{e-defpsi} along with the fact that $L^p(\Gamma,\nu)$ and $L^p(\Gamma,\mu)$ have comparable norms.
\item The condition \eqref{e-p0} is required only for the lower bound on $\psi_K$.
\item Let $\phi$ in Theorem \ref{t-main} be regularly varying with index $\beta>0$. If $\beta \in (0,2)$ we know matching two sided estimates on $k_n(x,y)$ for all $n \in \N$ and for all $x,y \in \Gamma$.
 Assume that $\phi(t)= ((1+t)l(t))^\beta $ where $l$ is a slowly varying function. The main result of \cite{MS1} states that
 \begin{equation} \label{e-full}
  k_n(x,y) \simeq \left( \frac{1}{ V(n^{1/\beta}l_{\#}(n^{1/\beta}))} \wedge \frac{n}{ V(d(x,y))\phi(d(x,y))} \right),
 \end{equation}
 where $l_\#$ is the de Bruijn conjugate of $l$.
 \item  We conjecture that the two-sided estimate \eqref{e-full} is true for any $\beta \in (0,\gamma)$, where $\gamma$ is the escape time exponent for the sub-Gaussian estimate in \eqref{e-uep} and \eqref{e-lep}.lem
 The proof of \eqref{e-full} in \cite{MS1} doesn't seem to work if $\beta \in [2, \gamma)$.
 In particular, the use of Davies' method to prove off-diagonal upper bounds does not seem to work directly.
  \item The conclusion of Theorem \ref{t-main} can be strengthened for random walks on groups for all values of $\beta$ ($\gamma$ is necessarily $2$ for random walks on groups). See \cite[Theorem 1.5]{SZ} for more.
 \item  Another intriguing question is to find matching two-sided estimates $k_n(x,y)$ for the case $\beta \ge \gamma$ for appropriate range of $d(x,y)$. In light of \cite[Theorem 1.5]{SZ} for random walks on groups, we conjecture that
 \[
  k_n(x,y) \simeq \frac{1}{V( \zeta(n))}
 \]
for all $n \in \N^*$ and for all $x,y \in \Gamma$ such that $d(x,y) \le \zeta(n)$.
\item It is a technically challenging open problem to replace the homogeneous volume doubling assumptions \eqref{e-vd} and \eqref{e-hom} by the more  general volume doubling assumption:
there exists $C_D>0$ such that $V(x,2r) \le C_D V(x,r)$ for all $x \in \Gamma$ and for all $r>0$.
\end{enumerate}
\end{remark}
Theorem \ref{t-main} indicates a possible moment threshold behavior. We define moment of random walk as follows.
\begin{definition}For a Markov operator $K$ on $\Gamma$ and any number $r>0$, we define the $r$-moment of random walk driven by $K$ as
\[
 M_{r,K}:= \sup_{x \in \Gamma} \EE^x d(X_0,X_1)^r= \sup_{x \in \Gamma} \left(K (d_x^r)\right)(x)
\]
where $(X_n)_{n \in \N}$ is a random walk driven by the Markov operator $K$ and $d_x^r:\Gamma \to \R$ denotes the function $y \mapsto (d(x,y))^r$.
\end{definition}
Here is a corollary of Theorem \ref{t-main} that illustrates moment threshold behavior of random walks. It states that the asymptotic behavior of $\psi_K$ is same as $\psi_P$ corresponding to the natural random walk
if and only if $K$ has finite $\gamma$-moment.
\begin{corollary} \label{c-moment}
  Let $(\Gamma,\mu)$ be an infinite, weighted graph satisfying \eqref{e-count}, \eqref{e-vd}, \eqref{e-hom}, \eqref{e-p0} and
 its heat kernel $p_n$ satisfies the sub-Gaussian bounds \eqref{e-uep} and \eqref{e-lep} with escape time exponent $\gamma$. Let $K$ be a Markov operator symmetric with respect to
 the measure $\mu$ satisfying \ref{e-ker}, where $\phi:[0,\infty) \to [1,\infty)$ is a continuous regularly varying function of positive index. Then the following are equivalent:
 \begin{enumerate}[(a)]
  \item $K$ has finite $\gamma$-moment, that is $M_{\gamma, K} < \infty$.
  \item There exists a constant $C>0$ such that
 \begin{equation} \label{e-srw}
  \frac{C^{-1}}{V(n^{1/\gamma})} \le \psi_K(n) \le \frac{C}{V( n^{1/\gamma})}
 \end{equation}
for all $n \in \N$.
 \end{enumerate}
\end{corollary}
\begin{remark}
 For random walks on groups one must have $\gamma=2$ and such a second moment threshold behavior is known in greater generality \cite[Theorem 1.4 and Corollary 1.5]{PS}.
 See \cite{BS1}, \cite{BS2} and \cite{SZ} for extensions and generalizations of such moment threshold behavior for random walks on groups. It is
 an interesting open problem to formulate and prove a $\gamma$-moment threshold in greater generality without the assumption \ref{e-ker}.
\end{remark}

\begin{proof}[Proof of Corollary \ref{c-moment}]
By Theorem \ref{t-main}, (b) holds if and only if $ \int_0^\infty \frac{s^{\gamma-1} \,ds}{\phi(s)} < \infty$.
Therefore (b) holds if and only if
\begin{equation} \label{e-cor1}
 \sum_{n=1}^\infty \frac{n^{\gamma-1}}{ \tilde \phi(n) } < \infty,
\end{equation}
where $\tilde \phi(x) = \sup_{t \in [0,x] } \phi(x)$.
The above statement follows
from Potter's bounds \cite[Theorem 1.5.6]{BGT}, continuity of $\phi$, Theorem 1.5.3 of \cite{BGT} and uniqueness of asymptotic inverse up to asymptotic equivalence.

 By \ref{e-ker} and Theorem 1.5.3 of \cite{BGT}, the condition $M_{\gamma,K} < \infty$ holds if and only if
 \begin{equation} \label{e-cor2}
  \sum_{y \in \Gamma} \frac{d(x,y)^\gamma}{ V(d(x,y)) \phi(d(x,y))} < \infty
 \end{equation}
for some fixed $x \in \Gamma$. It is well-known that the volume doubling property \eqref{e-vd} and \eqref{e-hom} implies a reverse volume doubling property which has the following consequence:
There exists an integer $A \in \N^*$ and $c_1 >0$ such that
\begin{equation} \label{e-cor3}
 V(x, Ar) - V(x,r) \ge c_1 V(r)
\end{equation}
  for all $r \ge 1/2$ (Proof of \cite[Proposition 3.3]{GHL} goes through with minor modifications).
There exists $c_2,c_3>0$ such that
\begin{align}
\nonumber \lefteqn{\sum_{y \in \Gamma} \frac{d(x,y)^\gamma}{ V(d(x,y)) \phi(d(x,y))} }\\
\nonumber  &\ge c_2 \sum_{n=0}^ \infty  \sum_{y \in B(x, A^{n+1}/2) \setminus B(x,A^n/2) }\frac{A^{n \gamma} }{ V(A^{n+1}/2) \tilde \phi(A^{n+1}/2)} \\
\label{e-cor4} & \ge  c_3 \sum_{n=0}^ \infty \frac{A^{n \gamma} }{ \tilde \phi(A^n)}
\end{align}
for all $x \in \Gamma$.

Now we show a reverse inequality of \eqref{e-cor4}.There exists $C_1,C_2>0$ such that
\begin{align}
\nonumber \sum_{y \in \Gamma} \frac{d(x,y)^\gamma}{ V(d(x,y)) \phi(d(x,y))}    &\le C_1 \sum_{n=0}^ \infty  \sum_{y \in B(x, 2^{n}) \setminus B(x,2^{n-1}) }\frac{2^{n \gamma} }{ V(2^{n-1}) \tilde \phi(2^{n-1})} \\
\label{e-cor5} & \le  C_2 \sum_{n=0}^ \infty \frac{2^{n \gamma} }{ \tilde \phi(2^n)}
\end{align}
for all $x \in \Gamma$.
The second line above follows from \eqref{e-vd} and Potter's bound \cite[Theorem 1.5.6]{BGT}.

To show (a) implies (b), we use \eqref{e-cor2}, \eqref{e-cor4}
and a generalization of Cauchy condensation test due to Schl\"{o}milch to obtain \eqref{e-cor1}.
To show (b) implies (a), we use \eqref{e-cor1}, Cauchy condensation test and \eqref{e-cor5} to obtain \eqref{e-cor2}  which  implies (b).
\end{proof}
Theorem \ref{t-main} and Corollary \ref{c-moment} suggests that for spaces with sub-Gaussian estimates and a scaling structure (for example regular fractals),  one might be able to formulate and prove a central limit theorem
with a  $\gamma+\epsilon$ moment condition.
\subsection{Analytic preliminaries on Markov operator and Dirichlet form}
Let  $(\Gamma,\mu)$ be a countable, weighted graph. Let $K$ be a Markov operator,
symmetric with respect to the measure $\mu$. Denote the kernel of the iterated operator $K^n$ with respect to $\mu$ by $k_n(x,y)$, that is $K^n f(x) = \sum_{y \in \Gamma} k_n(x,y) f(y) \mu(y)$.
We will collect some useful facts about the operator $K$. For any $p \in [1,\infty]$, we denote by $\norm{f}_p$ the norm of $f$ in $\ell^p(\Gamma,\mu)$ and by $\langle \cdot, \cdot \rangle$ the inner product in $\ell^2(\Gamma,\mu)$.
A fundamental property of $K$ is that it is a contraction in $\ell^p(\Gamma)$ for any $p \in [1,\infty]$, that is
\begin{equation*}\label{e-contr}
 \norm{Kf}_p \le \norm{f}_p
\end{equation*}
for all $p \in [1,\infty]$ and for all $f \in \ell^p(\Gamma,\mu)$. By the symmetry $k_1(x,y)=k_1(y,x)$, we have that $K$ is self-adjoint in $\ell^2(\Gamma,\mu)$, that is
\begin{equation}\label{e-sa}
 \langle Kf, g \rangle = \langle f, K g \rangle
\end{equation}
for all $f,g \in \ell^2(\Gamma,\mu)$.
For any $n \in \N$, we denote by $\E_{K^n}(f,f)= \langle (I- K^n)f,f \rangle$ the Dirichlet form associated with $K^n$.

The following useful lemma compares Dirichlet form of a Markov operator $K$ with its iterated power $K^{n}$.
\begin{lemma}[Folklore]\label{l-dircomp}
 Let $K$ be a Markov operator on $\Gamma$ symmetric with respect to the measure $\mu$. Then for any $f \in \ell^2(\Gamma,\mu)$ and for any $n \in \N^*$
 \begin{equation} \label{e-dircomp}
 \E_{K^{n}}(f,f) \le n \E_{K}(f,f).
 \end{equation}
\end{lemma}
\begin{proof}
 We verify this using spectral theory. Let $E_\lambda$ be the spectral resolution of $K$. Therefore
 \[
   \E_{K^{n}}(f,f) - n \E_{K}(f,f)= \int_{-1}^1  (1 - \lambda^{n} - n + n \lambda) d E_\lambda(f,f).
 \]
The result follows from the observation that $1 - \lambda^{n} - n + n \lambda \le 0$ for all $\lambda \in [-1,1]$ and for all $n \in \N^*$.
\end{proof}
\begin{lemma}[Folklore]\label{l-noninc}
  Let $K$ be a Markov operator on $\Gamma$ symmetric with respect to the measure $\mu$ and let  $f \in \ell^2 (\Gamma,\mu)$ be a non-zero function. Then the function $i \mapsto \norm{K^i f}_2/\norm{K^{i-1}f}_2$ is non-decreasing.
\end{lemma}
\begin{proof}
 We use self-adjointness \eqref{e-sa} and Cauchy-Schwarz inequality to obtain
 \[
  \norm{K^if}_2^2 = \langle K^{i-1}f ,K^{i+1}f\rangle \le \norm{K^{i-1}f}_2 \norm{K^{i+1}f}_2,
 \]
which gives the desired result.
\end{proof}

\section{Pseudo-Poincar\'{e} inequality using Discrete subordination}

Pseudo-Poincar\'{e} inequality provides an efficient way to  prove Nash inequality which in turn gives upper bounds on $\psi_K(n)$.
For a function $f:\Gamma \to \R$ and $R>0$, we define a function $f_R:\Gamma \to \R$ by
\[
 f_R(x) := \frac{1}{V(x,R)} \sum_{y \in B(x,R)} f(y) \mu(y).
\]
In other words, $f_R(x)$ is the $\mu$-average of $f$ in $B(x,R)$.
The main result of the section is the following pseudo-Poincar\'{e} inequality.
\begin{proposition}[Pseudo-Poincar\'{e} inequality]\label{p-pp}
Let $(\Gamma,\mu)$ be a weighted graph satisfying \eqref{e-count}, \eqref{e-vd}, \eqref{e-hom} and suppose that
 its heat kernel $p_n$ satisfies the sub-Gaussian bounds \eqref{e-uep} and \eqref{e-lep} with escape time exponent $\gamma$. Let $K$ be a Markov operator symmetric with respect to
 the measure $\mu$ satisfying \ref{e-ker}, where $\phi:[0,\infty) \to [1,\infty)$ is a continuous regularly varying function of positive index.
 There exists a constant $C>0$ such that
\begin{equation}
 \label{e-pp} \norm{f-f_R}_2^2 \le C \left( \frac{R^\gamma}{\int_0^R \frac{s^{\gamma-1}\,ds}{\phi(s)}}\right) \E_K(f,f)
\end{equation}
for all $R>0$ and for all $f \in \ell^2(\Gamma,\mu)$.
\end{proposition}
We introduce a discrete subordination of the natural random walk on $(\Gamma,\mu)$ whose kernel is comparable to the kernel of $K$ in Proposition \ref{p-sub}.
We introduce a new Markov operator \begin{equation}\label{e-defQ}
                                 Q:= \frac{1}{2}(P+P^2)
                                \end{equation}
                                 which has a symmetric kernel $q(x,y)=\frac{1}{2}( p_1(x,y)+p_2(x,y))$ with respect to $\mu$. Let $q_k$ denote the kernel of the Markov operator $Q^k$. For a Markov operator $Q^k$,
let $\E_{Q^k}(f,f):= \langle (I-Q^k)f,f \rangle$ denote the corresponding Dirichlet form. Let $R_Q$ denote the resistance defined using the Dirichlet form $\E_Q$. We will now compare kernels of $P^k$ and $Q^k$.
\begin{remark}
 The advantage of working with the kernel $q_n$ is that it satisfies as stronger sub-Gaussian lower estimate \eqref{e-lq} in comparison to \eqref{e-lep} satisfied by $p_n$.
 This makes subordination of kernel $Q$ preferable(as opposed to $P$) and technically easier.
\end{remark}
\begin{lemma}\label{l-qk}
 The kernel $q_k$ satisfies the following improved sub-Gaussian estimates:  there exist constants $c,C>0$ such that, for all $x,y \in \Gamma$
\begin{equation}\label{e-uq}
 q_n(x,y) \le \frac{C}{V(n^{1/\gamma})} \exp \left[ - \left( \frac{d(x,y)^\gamma}{Cn} \right)^{\frac{1}{\gamma-1}} \right], \forall n \ge 1
\end{equation}
and
\begin{equation}\label{e-lq}
 q_n(x,y) \ge \frac{c}{V(n^{1/\gamma})}  \exp \left[ - \left( \frac{d(x,y)^\gamma}{cn} \right)^{\frac{1}{\gamma-1}} \right], \forall n \ge 1 \vee d(x,y).
\end{equation}
\end{lemma}
\begin{proof}
 Observe that  $q_n(x,y)= \sum_{k=0}^n 2^{-n} \binom{n}{k} p_{n+k}(x,y)$. This along with \eqref{e-uep}, \eqref{e-vd} gives the desired upper bound \eqref{e-uq}.

 Note that, there exists $C_1>1$ such that
 \begin{equation} \label{e-qk1}
C_1^{-1} \le  \frac{\binom{n}{k}}{\binom{n}{k+1} } \le C_1
 \end{equation}
for all $n \in \N^*$ and for all $k \in \N$ such that $ \lfloor \frac{n}{4} \rfloor \le k \le \lfloor\frac{3n}{4}\rfloor$.
There exists $c_1,c_2>0$
\begin{align*}
 q_n(x,y)&\ge \sum_{k=\lfloor \frac{n}{4} \rfloor }^{ \lfloor\frac{3n}{4}\rfloor+1} 2^{-n} \binom{n}{k} p_{n+k}(x,y) \\
 &  \ge  2^{-n-1} C_1^{-1} \sum_{k=\lfloor \frac{n}{4} \rfloor }^{ \lfloor\frac{3n}{4}\rfloor} \binom{n}{k}( p_{n+k}(x,y) + p_{n+k+1}(x,y)) \\
 & \ge  2^{-n-1}c_1 C_1^{-1} C_D^{-1} \frac{1}{V(n^{1/\gamma})} \exp \left[ - \left( \frac{d(x,y)^\gamma}{c_1 n} \right)^{\frac{1}{\gamma-1}} \right]\sum_{k=\lfloor \frac{n}{4} \rfloor }^{ \lfloor\frac{3n}{4}\rfloor} \binom{n}{k} \\
 &\ge  c_2 \frac{1}{V(n^{1/\gamma})} \exp \left[ - \left( \frac{d(x,y)^\gamma}{c_1 n} \right)^{\frac{1}{\gamma-1}} \right]
\end{align*}
for all $x,y \in \Gamma$ and for all $n \in \N$ such that $n \ge 1 \vee d(x,y)$. The second line above follows from \eqref{e-qk1}, the third line follows from \eqref{e-lep} and \eqref{e-vd} and
the last line follows from weak law of large numbers.
\end{proof}
The operators $P$ and $Q$ have comparable Dirichlet forms and resistances.
\begin{lemma}\label{l-res}
 The resistances $R_Q$ and $R_P$ are comparable by the following inequality
 \[
\frac{1}{2}  R_P(f,f) \le R_Q(A,B) \le 2 R_P(A,B)
 \]
for all subsets $A,B \subset \Gamma$.
\end{lemma}
\begin{proof}
 It suffices to compare the corresponding Dirichlet forms $\E_Q$ and $\E_P$.
 Note that $\E_Q(f,f) = \frac{1}{2} (\E_P(f,f) + \E_{P^2}(f,f)) \ge \frac{1}{2}\E_P(f,f)$. However by Lemma \ref{l-dircomp}, we have
 \[
  \E_Q(f,f) = \frac{1}{2} (\E_P(f,f) + \E_{P^2}(f,f)) \le \frac{3}{2} \E_P(f,f) \le 2 \E_P(f,f).
 \]
\end{proof}
 We have the following pseudo-Poincar\'{e} inequality for iterated powers of $Q$.
\begin{lemma} \label{l-pp0}
Under the assumptions of Proposition \ref{p-pp}, there exists $C_1>0$ such that
 \[
  \norm{f-f_R}^2_2 \le C_1 \left( \frac{R}{k} \right)^\gamma \E_{Q^{2 \flr{k^\gamma}}}(f, f)
 \]
for all $f \in \ell^2(\Gamma,\mu)$ and for all $k \in \N$ and $R \in \R$ satisfying $1 \le k \le R$.
\end{lemma}
\begin{proof}
There exists $C_2 >0$ such that
 \begin{align}
  \nonumber \norm{f-f_R}^2 & \le   \sum_{x \in \Gamma} \sum_{y \in B(x,R)} \frac{(f(x)-f(y))^2}{V(x,R)} \mu(y) \mu(x) \\
 \nonumber &\le C_2 \sum_{x \in \Gamma} \sum_{y \in \Gamma} (f(x)-f(y))^2 q_{2 \flr{R^\gamma}}(x,y) \mu(y)\mu(x) \\
  \label{e-p01} & =  2 C_2 \left(  \norm{f}_2^2 - \norm{Q^{\flr{R^\gamma}}f}_2^2 \right).
 \end{align}
The first line follows from Jensen's inequality, the second line follows from the lower bound \eqref{e-lq} of Lemma \ref{l-qk}, \eqref{e-vc} and \eqref{e-hom}, the last line follows from the $\mu$-symmetry of $Q$.
Since $Q$ is a contraction on $\ell^2(\Gamma,\mu)$, we have
\begin{equation} \label{e-p02}
 \norm{f}_2^2 - \norm{Q^{\flr{R^\gamma}}f}_2^2 \le  \norm{f}_2^2 - \norm{Q^{l \flr{k^\gamma}}f}_2^2 = \sum_{m=0}^{l-1} \left( \norm{Q^{m \flr{k^\gamma}} f}_2^2-  \norm{Q^{(m+1) \flr{k^\gamma}} f}_2^2 \right)
\end{equation}
where $l = \cil{\flr{R^\gamma}/\flr{k^\gamma}}$.

Since $Q$ is a contraction on $\ell^2(\Gamma,\mu)$, we have
\begin{align}
 \nonumber \norm{Q^{(m+1) \flr{k^\gamma}} f}_2^2-  \norm{Q^{(m+2) \flr{k^\gamma}} f}_2^2 &= \norm{Q^{(m+1)\flr{k^\gamma}}( I - Q^{2\flr{k^\gamma}})^{1/2}f}_2^2 \\
&\le \norm{Q^{m\flr{k^\gamma}}( I - Q^{2\flr{k^\gamma}})^{1/2}f}_2^2 \nonumber \\
&= \norm{Q^{m \flr{k^\gamma}} f}_2^2-  \norm{Q^{(m+1) \flr{k^\gamma}} f}_2^2. \label{e-p03}
\end{align}
By \eqref{e-p02} and \eqref{e-p03}, we get
\begin{equation} \label{e-p04}
 \norm{f}_2^2 - \norm{Q^{\flr{R^\gamma}}f}_2^2 \le  l  \left(  \norm{ f}_2^2-  \norm{Q^{ \flr{k^\gamma}} f}_2^2 \right) \le 4 \frac{R^\gamma}{k^\gamma} \left(  \norm{ f}_2^2-  \norm{Q^{ \flr{k^\gamma}} f}_2^2 \right).
\end{equation}
Combining \eqref{e-p01} and \eqref{e-p04} gives the desired inequality.
\end{proof}
The following subordinated kernel satisfying \ref{e-ker} is a useful tool to study the behavior of long range random walks.
\begin{proposition} \label{p-sub}
 Let $\phi:[0,\infty) \to [1,\infty)$ be a continuous regularly varying function of positive index.
 Let $(\Gamma,\mu)$  be a weighted graph satisfying the assumptions of Proposition \ref{p-pp} and let $Q$ be defined by \eqref{e-defQ}. Define the subordinated Markov kernel
 \begin{equation}\label{e-qp}
  Q_\phi := \sum_{n=1}^\infty c_\phi \frac{1}{n\phi(n)} Q^{ 2 \flr{n^\gamma}}
 \end{equation}
where $c_\phi= \left( \sum_{n=1}^\infty \frac{1}{n\phi(n)} \right)^{-1}$. Then $Q_\phi$ has a symmetric kernel $q_\phi$ with respect to $\mu$ and
there exists $C >0$ such that
\begin{equation} \label{e-qphi}
C^{-1} \frac{1}{V(d(x,y)) \phi(d(x,y))} \le q_\phi(x,y)= q_\phi(y,x) \le  C \frac{1}{V(d(x,y)) \phi(d(x,y))}.
\end{equation}
for all $x,y \in \Gamma$.
In other words, $Q_\phi$ satisfies \ref{e-ker}.
\end{proposition}
\begin{proof}
The symmetry of $Q_\phi$ follows from the symmetry of $Q$ since
\begin{equation}
 \label{e-sub0} q_\phi(x,y) := \sum_{n=1}^\infty c_\phi \frac{1}{n\phi(n)} q_{ 2 \flr{n^\gamma}}(x,y).
\end{equation}

Let $\phi$ be regularly varying of index $\beta>0$.
By Potter's bounds \cite[Theorem 1.5.6]{BGT} and using that $\phi$ is a positive continuous function, there exists $C_1>0$ such that
 \begin{equation}
  \label{e-pot} \frac{\phi(s)}{\phi(t)} \le C_1 \max \left( \left(\frac{s}{t}\right)^{3\beta/2}, \left(\frac{s}{t}\right)^{\beta/2}\right)
 \end{equation}
for all $s,t \in [1,\infty)$.

It suffices to assume that $x,y \in \Gamma$ and $x \neq y$. The case $x=y$ follows trivially from Lemma \ref{l-qk}.
Combining $n^\gamma/2 \le \flr{n^\gamma} \le n^\gamma$, \eqref{e-sub0}, \eqref{e-vd} and \eqref{e-uq} of Lemma \ref{l-qk}, there exists $C_2>0$ such that
\begin{equation} \label{e-sub1}
 q_\phi(x,y) \le\sum_{n=d(x,y)+1}^{\infty} \frac{ C_2}{n \phi(n)V(n)} + \sum_{n=1}^{d(x,y)} \frac{ C_2}{n \phi(n) V(n)} \exp \left[ - \left( \frac{d(x,y)}{C_2 n}\right)^{\gamma/(\gamma-1)}\right]
\end{equation}
for all $x,y \in \Gamma$ with $x \neq y$.
We bound the first term in \eqref{e-sub1} by
\begin{align}
 \nonumber \sum_{n=d(x,y)+1}^{\infty} \frac{ 1}{n \phi(n)V(n)} &\le \frac{1}{V(d(x,y))} \sum_{n=d(x,y)+1}^\infty \frac{ 1}{n \phi(n)} \\
\nonumber &\le C_3 \frac{1}{V(d(x,y))} \int_{d(x,y)}^\infty \frac{ ds}{s \phi(s)} \\
\label{e-sub2} &\le C_4  \frac{1}{V(d(x,y))\phi(d(x,y))}
\end{align}
where $C_3,C_4>0$ are constants. In the first line above, we used that $V$ is non-decreasing. The second line above follows from \eqref{e-pot} and the third line follows from \cite[Proposition 1.5.10]{BGT}.

Let $ 1 \le n \le d(x,y)$. To estimate second term in \eqref{e-sub1}, we use \eqref{e-pot} and \eqref{e-vc} to obtain
\begin{align}
 \nonumber \frac{1}{n \phi(n) V(n) } &=  \frac{1}{d(x,y) \phi(d(x,y)) V(d(x,y))} \frac{d(x,y) \phi(d(x,y)) V(d(x,y))}{n \phi(n) V(n) } \\
 \label{e-sub3} &\le \frac{ C_1 C_D}{d(x,y) \phi(d(x,y)) V(d(x,y))} \left(  \frac{d(x,y)}{n}\right)^{\alpha+ ((3\beta)/2)+ 1}.
\end{align}
Since the function $t\mapsto t^{\alpha+ ((3\beta)/2)+ 1} \exp \left[ - (C_2^{-1}t)^{\gamma/(\gamma-1)} \right]$ is uniformly bounded (by say $C_5$) in $[1,\infty)$, by \eqref{e-sub3}, there exists a constant $C_6>0$ such that
\begin{equation} \label{e-sub4}
 \sum_{n=1}^{d(x,y)} \frac{ C_2}{n \phi(n) V(n)} \exp \left[ - \left( \frac{d(x,y)}{C_2 n}\right)^{\gamma/(\gamma-1)}\right] \le \frac{C_6}{V(d(x,y)) \phi(d(x,y))}
\end{equation}
for all $x,y \in \Gamma$ with $x \neq y$. Combining \eqref{e-sub1}, \eqref{e-sub2} and \eqref{e-sub4} gives the desired upper bound in \eqref{e-qphi}.

For the lower bound in \eqref{e-qphi}, we use \eqref{e-sub0}, \eqref{e-lq} of Lemma \ref{l-qk} along with \eqref{e-vc} to obtain, a constant $c_1>0$ such that
\begin{align}
 \nonumber q_\phi(x,y) & \ge   \sum_{n=d(x,y)}^{2d(x,y)} \frac{c_\phi}{n\phi(n)} q_{2\flr{n^\gamma}} (x,y) \\
 \nonumber  & \ge  \sum_{n=d(x,y)}^{2d(x,y)} \frac{c_1}{n\phi(n)V(n)} \\
 \nonumber & \ge  \frac{C_D^{-1} c_1 }{2d(x,y) V(d(x,y)) \phi(d(x,y))} \sum_{n=d(x,y)}^{2d(x,y)} \frac{1}{3 C_1}
\end{align}
for all $x,y \in \Gamma$ with $x \neq y$. In the last line, we used, \eqref{e-vd}, $n^{-1} \ge (2d(x,y))^{-1}$ and the Potter's bound \eqref{e-pot}.
\end{proof}

\begin{proof}[Proof of Proposition \ref{p-pp}]
 By Proposition \ref{p-sub}, the Markov operators $K$ and $Q_\phi$ have comparable Dirichlet forms. Hence it suffices to consider the case $K=Q_\phi$.
 If $R <1$, then $f \equiv f_R$  which in turn implies the pseudo-Poincar\'{e} inequality \eqref{e-pp}.

 Hence we assume that $R \ge 1$.
There exists $c_1>0$ such that
 \begin{align}
 \nonumber \E_{Q_\phi}(f,f) &= c_\phi \sum_{k=1}^\infty \frac{1}{ k \phi(k)} \E_{Q^{2\flr{k^\gamma}}} (f,f) \\
 \nonumber & \ge c_\phi C_1^{-1} \norm{f-f_R}_2^2 R^{-\gamma}\sum_{k=1}^ {\flr{R}} \frac{k^{\gamma-1}}{  \phi(k)} \\
 & \ge  c_1 \norm{f-f_R}_2^2 R^{-\gamma} \int_0^R \frac{s^{\gamma-1}\,ds }{\phi(s)}
 \end{align}
 for all $f \in \ell^2(\Gamma,\mu)$ and for all $R >0$ which is the desired inequality.
 In the second line above, we used Lemma \ref{l-pp0} and in the last line we used that $\phi$ is a positive continuous regularly varying function which satisfies the  Potter's bound \eqref{e-pot}.
\end{proof}

\section{Nash inequality and Ultracontractivity.}
In this section, we use pseudo-Poincar\'{e} inequality \eqref{e-pp} to obtain a Nash inequality and on-diagonal upper bounds.
A polished treatment of the relationship between Nash inequalities and ultracontractivity is presented in \cite{Cou}.
It is well-known that pseudo-Poincar\'{e} inequality along with assumptions on volume growth gives a Sobolev-type inequality (see \cite[Theorem 2.1]{S-C} for an early reference to this approach).

The following function $\eta$ which appears in \eqref{e-pp} plays a crucial role in this work. Define the function $\eta:[0,\infty) \to (0,\infty)$
\begin{equation} \label{e-defeta}
\eta(R) :=  \frac{R^\gamma}{\int_0^R \frac{s^{\gamma-1}\,ds}{\phi(s)}}
\end{equation}
for $R>0$ and $\eta(0)= \gamma \phi(0)$ so that $\eta$ is a continuous function. We also need the following modification of $\eta$ defined as $\tilde \eta:[0,\infty) \to (0,\infty)$
\begin{equation} \label{e-defteta}
\tilde \eta(R) :=  \sup \{ \eta(t) : t \in [0,R] \}
\end{equation}
so that $\tilde \eta$ is a non-decreasing function. It is known that \cite[Theorem 1.5.3]{BGT} $\tilde \eta$ is asymptotically equivalent to $\eta$, that is $\lim_{t \to \infty} \tilde \eta(t)/ \eta(t)=1$.
If $\phi$ is regularly varying with positive index, so is $\eta$. We now compute the index of $\eta$ and list some of its basic properties.
\begin{lemma} \label{l-eta}
 If $\phi:[0,\infty) \to [1,\infty)$ is a continuous regularly varying function with index $\beta >0$, then
 \begin{enumerate}[(a)]
  \item  The function $\eta$ defined by \eqref{e-defeta} is continuous, positive and  regularly varying with index $\beta \wedge \gamma$.
  \item There exists $C_1>0$ such that $\eta(x) \le C_1 \phi(x)$ for all $x \ge 0$.
  \item  The function $\eta$ has an asymptotic inverse $\zeta:[0,\infty) \to [1,\infty)$ satisfying the following properties: $\zeta$ is continuous, non-decreasing and regularly varying with index $1/(\beta \wedge \gamma)$.
  Moreover, there exists $C>0$ such that
  \begin{equation} \label{e-inv}
  C^{-1} t \le \zeta(\eta(t)) \le \zeta(\tilde \eta(t)) \le C t \hspace{1cm} \mbox{and} \hspace{1cm} C^{-1} t \le  \eta(\zeta(t)) \le \tilde \eta(\zeta(t)) \le C t
  \end{equation}
for all $t \ge 1$.
 \end{enumerate}
\end{lemma}
\begin{proof}
(a) and (b): The cases $\beta< \gamma$, $\beta=\gamma$ and $\beta > \gamma$ follow from Proposition 1.5.8, Proposition 1.5.9a and Proposition 1.5.10  in \cite{BGT} respectively.

(c) The existence of an asymptotic inverse which is regularly varying of index $1/(\beta \wedge \gamma)$  follows from (a) and  \cite[Proposition 1.5.12]{BGT}.
The fact that $\zeta$ can be chosen to be continuous, bounded below by $1$ and
non-decreasing  follows from Theorem 1.8.2, Proposition 1.5.1 and Theorem 1.5.3 of \cite{BGT} respectively. The existence of $C>0$ satisfying \eqref{e-inv} follows from the definition of asymptotic inverse and
continuity of $\zeta$,$\eta$ and $\tilde \eta$ and $\lim_{t \to \infty} \tilde \eta(t)/\eta(t) =1$.
\end{proof}

\begin{theorem}[Nash inequality] \label{t-nash}
 Let  $\phi:[0,\infty) \to [1,\infty)$ be a  continuous, regularly varying  function of positive index.
Let $K$ be Markov operator satisfying \ref{e-ker} with symmetric kernel $k$ with respect to the measure $\mu$. Then there exist constants $C_1,C_2>0$ such that
\begin{equation}\label{e-nash}
 \norm{f}_2^2 \le C_1 \E_{K^2}(f,f) \tilde \eta \left( V^{-1}\left(C_2 \frac{\norm{f}_1^2}{\norm{f}_2^2} \right)\right)
\end{equation}
for all $f \in \ell^1(\Gamma,\mu)$, where $\tilde \eta$ is given by \eqref{e-defeta} and \eqref{e-defteta}.
\end{theorem}
\begin{proof}
 Let $R>0$ and $f \in \ell^{1}(\Gamma,\mu)$.

By \eqref{e-hom} and triangle inequality, we have \[ \norm{f_R}_\infty \le C_h \norm{f}_1/ V(R) \hspace{1cm} \operatorname{and}\hspace{1cm}  \norm{f_R}_1 \le C_h^2  \norm{f}_1.\]
 Hence by H\"{o}lder's inequality
 \begin{equation} \label{e-n1}
  \norm{f_R}_2^2 \le \norm{f_R}_\infty \norm{f_R}_1 \le C_h^3  \frac{ \norm{f}_1^2}{V(R)}
 \end{equation}
for all $f \in \ell^{1}(\Gamma,\mu)$ and for all $R> 0$.  By \eqref{e-n1} and Proposition \ref{p-pp}, there exists $C_3>0$ such that
\begin{align}
 \nonumber \norm{f}_2^2 &\le 2 \norm{f-f_R}_2^2 + 2 \norm{f_R}_2^2 \\
\nonumber &\le C_3 \left( \eta(R) \E_K(f,f) + \frac{\norm{f}_1^2}{V(R)} \right) \\
\label{e-n2} &\le C_3 \left( \tilde{\eta}(R) \E_K(f,f) + \frac{\norm{f}_1^2}{V(R)} \right).
\end{align}
To minimize \eqref{e-n2}, we want to choose $R=R_0 >0$ such that $ \left( \tilde \eta (R_0)V(R_0) \right)^{-1} \simeq \E_K(f,f)/\norm{f}_1^2$.

Note that $R \mapsto  \left(\tilde{\eta}(R)V(R) \right)^{-1}$ is a strictly decreasing continuous function with
\[
 \lim_{R \to 0^+} \left( \tilde \eta (R)V(R) \right)^{-1}  =  \left(\eta (0)V(0)\right)^{-1} \hspace{5mm} \operatorname{and} \hspace{5mm}  \lim_{R \to \infty} \left( \tilde \eta (R)V(R) \right)^{-1}  = 0.
\]
Therefore the equation
\begin{equation} \label{e-n3}
 \left( \tilde \eta (R)V(R) \right)^{-1}  = t
\end{equation}
has an unique solution for all $ t \in \left(0, \left(\eta (0)V(0)\right)^{-1} \right]$.

Since $K$ is a contraction in $\ell^2(\Gamma,\mu)$, we have
\[
 \E_K(f,f) = \langle (I-K)f,f \rangle  \le \norm{f}_2^2 + \abs{\langle Kf,f \rangle } \le \norm{f}_2^2 + \norm{f}_2 \norm{Kf}_2 \le 2 \norm{f}_2^2.
\]
By \eqref{e-count} and using $\ell^p$ inequalities for counting measure, we have $\norm{f}_1^2 \ge C_\mu^{-3} \norm{f}_2^2$. Combining these observations gives
\begin{equation} \label{e-n4}
 \E_K(f,f)/\norm{f}_1^2 \le 2 C_\mu^3
\end{equation}
for all $f \in \ell^{1}(\Gamma,\mu)$.
By \eqref{e-n3} and \eqref{e-n4}, for any $f \in \ell^1(\Gamma,\mu)$ with $f \neq 0$, there exists an unique solution $R_0$ to the equation
\begin{equation}\label{e-n5}
  \left( \tilde \eta (R_0)V(R_0) \right)^{-1}  = c_1 \frac{\E_K(f,f)}{\norm{f}_1^2},
\end{equation}
where $c_1=\left(2 C_\mu^3 \eta (0)V(0)\right)^{-1}$. Substituting the above solution $R_0$ in \eqref{e-n2} gives
$\norm{f}_2^2 \le C_3 (1+c_1^{-1}) \norm{f}_1^2 / V(R_0) $ or equivalently,
\begin{equation} \label{e-n5a}
 R_0 \le V^{-1} \left( C_2 \norm{f}_1^2/ \norm{f}_2^2 \right)
\end{equation}
where $C_2:= C_3 (1+c_1^{-1}) $.
Since $\tilde \eta$ is a non-decreasing function, by \eqref{e-n5} and \eqref{e-n5a} we have
\[
 \frac{\norm{f}_1^2}{\E_K(f,f)} \le c_1 C_2 \frac{\norm{f}_1^2}{\norm{f}_2^2} \tilde \eta \left( V^{-1}\left(C_2 \frac{\norm{f}_1^2}{\norm{f}_2^2} \right)\right).
\]
 Hence we obtain the Nash inequality
\begin{equation}\label{e-n6}
 \norm{f}_2^2 \le c_1 C_2 \E_K(f,f) \tilde \eta \left( V^{-1}\left(C_2 \frac{\norm{f}_1^2}{\norm{f}_2^2} \right)\right).
\end{equation}
By \ref{e-ker} and \eqref{e-count}, there exists $\alpha>0$ such that $\inf_{x \in \Gamma} k_1(x,x) \mu(x) \ge \alpha$.  Since $k_2(x,y) \ge k_1(x,y) k_1(y,y) \mu(y) \ge \alpha k_1(x,y)$, we have
\[
 \E_K(f,f) \le \alpha^{-1} \E_{K^2} (f,f)
\]
for all $f \in \ell^2(\Gamma,\mu)$. This along with \eqref{e-n6} gives the desired Nash inequality.
\end{proof}

\begin{theorem}[Ultracontractivity]\label{t-ultra}
  Let $(\Gamma,\mu)$ be a weighted graph satisfying \eqref{e-count}, \eqref{e-vd}, \eqref{e-hom} and
 its heat kernel $p_n$ satisfies the sub-Gaussian bounds \eqref{e-uep} and \eqref{e-lep} with escape time exponent $\gamma$. Let $K$ be a Markov operator symmetric with respect to
 the measure $\mu$ satisfying \ref{e-ker}, where $\phi:[0,\infty) \to [1,\infty)$ is a continuous regularly varying function of positive index. Then there exists a constant $C>0$ such that
 \begin{equation*}
   \psi_K(n) \le \frac{C}{V(\zeta(n))}
 \end{equation*}
for all $n \in \N$, where $\zeta:[0,\infty) \to [1,\infty)$ is a continuous non-decreasing function which is an asymptotic inverse of
 $t \mapsto t^{\gamma}/ \int_0^t \frac{s^{\gamma-1}\, ds}{\phi(s)}$.
\end{theorem}
\begin{proof}
 Let $\mu_*= \inf_{x \in \Gamma} \mu(x)$. Define $g:(0,1/\mu_*] \to [0,\infty)$ by
 \[
  g(t):= \int_{\mu_*}^{1/t} C_1 \tilde \eta \left( V^{-1} (C_2 s)\right) \frac{ds}{s}
 \]
and $m:[0,\infty) \to (0,1/\mu_*]$ as the inverse of $g$, where $C_1,C_2$ are constants from \eqref{e-nash}.
Since  $g$ is a decreasing, surjective, continuous function, so is $m$.
Observe that we can increase the constant $C_2$ in \eqref{e-nash} without affecting the Nash inequality. We choose $C_2$ such that $C_2 \ge V(1)/\mu_*$, so that
\begin{equation}\label{e-ult1}
V^{-1}(C_2s) \ge 1
\end{equation}
for all $s \ge \mu_*$.

By a standard ultracontractivity estimate using Nash inequality \eqref{e-nash} (see \cite[ Theorem 2.2.1]{PSu} or \cite[Proposition IV.1]{Cou}), we obtain
\begin{equation}\label{e-ult2}
 \psi_K(n) \le m(n)
\end{equation}
for all $n \in \N^*$.

We now estimate the functions $g(t)$ and its inverse $m(t)$. For $t^{-1} \ge \mu_*$, choose $L \in \N$ such that $C_D^L \mu_* \in [t^{-1},C_D t^{-1})$.
We have
\begin{align}
 \nonumber g(t) &\le \int_{\mu_*}^{C_D^L \mu_* } C_1 \tilde \eta \left( V^{-1} (C_2 s)\right) \frac{ds}{s} = C_1 \int_{C_2 \mu_*}^{C_D^L C_2 \mu_* }  \tilde \eta \left( V^{-1} ( s)\right) \frac{ds}{s} \\
 \nonumber &\le C_1  \sum_{k=1}^L \int_{C_D^{k-1} C_2 \mu_*}^{C_D^k C_2 \mu_*} \frac{ \tilde \eta(V^{-1}(s)) \, ds}{s} \\
 & \le C_1 \sum_{k=1}^L \frac{ \tilde \eta(V^{-1}(C_D^k C_2 \mu_*)) }{ C_D^{k-1} C_2 \mu_* } \left( (C_D-1) C_D^{k-1}  C_2 \mu_*\right) \nonumber\\
 &\le  C_3 \sum_{k=1}^L  \tilde \eta(V^{-1}(C_D^k C_2 \mu_*)) \label{e-ult3}
\end{align}
where $C_3= C_1 (C_D-1)$.
In the third line above, we used that $\tilde \eta \circ V^{-1}$ is a non-decreasing function.

By Lemma \ref{l-eta} and \cite[Theorem 1.5.3]{BGT}, $\tilde \eta$ is regularly varying of positive index.
Hence by Potter's bounds \cite[Theorem 1.5.6]{BGT} and using that $\tilde \eta$ is a positive continuous function, there exists $C_4>1$, $\beta_1>\beta_2>0$ such that
\begin{equation}\label{e-ult4}
 \frac{\tilde \eta (s)}{\tilde \eta (t)} \le C_4 \max \left( \left( \frac{s}{t} \right)^{\beta_1}, \left( \frac{s}{t} \right)^{\beta_2} \right)
\end{equation}
for all $s,t \in [1,\infty)$.
By \eqref{e-vd}, \eqref{e-ult1} and \eqref{e-ult4} , we get
\begin{equation} \label{e-ult5}
 \tilde \eta(V^{-1}(C_D^k C_2 \mu_*)) \le  \tilde \eta(2^{k-L}V^{-1}(C_D^L C_2 \mu_*)) \le C_4 2^{\beta_2(k-L)} \tilde \eta(V^{-1}(C_D^L C_2 \mu_*))
\end{equation}
for all $k=1,2,\ldots,L$.
By \eqref{e-ult3} and \eqref{e-ult5},
\begin{equation*} \label{e-ult6}
g(t) \le C_5  \tilde \eta(V^{-1} (C_D C_2 /t))
\end{equation*}
for all $t \le \mu_*^{-1}$, where $ C_5:=C_3C_4(1-2^{-\beta_2})^{-1}$. Therefore
\[
 t =g(m(t)) \le C_5 \tilde \eta(V^{-1} (C_D C_2 /m(t)))
\]
for all $t \ge 0$.

We use an asymptotic inverse $\zeta$ of the function $\eta$ as described in Lemma \ref{l-eta}.
Hence by Potter's theorem \cite[Theorem 1.5.6]{BGT}) and \eqref{e-inv} , there exists $C_6,C_7>0$ such that
\begin{equation}
  \zeta(t) \le C_6 \zeta(t/C_5) \le C_6 \zeta(\eta(V^{-1} (C_D C_2 /m(t)))) \le C_7 V^{-1} (C_D C_2 /m(t))
\end{equation}
for all $t \ge 1$.
By \eqref{e-vc}, there exists $C_8>0$ such that
\begin{equation}
 m(t) \le C_D C_2 / V( \zeta(t)/C_7) \le \frac{C_8}{V(\zeta(t))}.
\end{equation}
The conclusion follows from \eqref{e-ult2}.
\end{proof}
\section{Lower bound on \texorpdfstring{$\psi_K$}{psik}}
The lower bound on $\psi_K$ follows from a test function argument due to Coulhon and Grigor'yan \cite[Theorem 4.6]{CG}.
However we need a good test function for that argument to work. Such a test function can be obtained from the resistance estimate in \eqref{e-res}.
\begin{theorem} \label{t-lb}
   Let $(\Gamma,\mu)$ be a weighted graph satisfying \eqref{e-count}, \eqref{e-vd}, \eqref{e-hom}, \eqref{e-p0} and
 its heat kernel $p_n$ satisfies the sub-Gaussian bounds \eqref{e-uep} and \eqref{e-lep} with escape time exponent $\gamma$. Let $K$ be a Markov operator symmetric with respect to
 the measure $\mu$ satisfying \ref{e-ker}, where $\phi:[0,\infty) \to [1,\infty)$ is a continuous regularly varying function of positive index. Then there exists a constant $c>0$ such that
 \begin{equation*}
   \psi_K(n) \ge \frac{c}{V(\zeta(n))}
 \end{equation*}
for all $n \in \N$, where $\zeta:[0,\infty) \to [1,\infty)$ is a continuous non-decreasing function which is an asymptotic inverse of
$t \mapsto t^\gamma/ \int_0^t \frac{s^{\gamma -1} \,ds}{\phi(s)}$.
\end{theorem}

\begin{proof}
 By Lemma \ref{l-noninc}, we have
 \begin{equation} \label{e-lb1}
  \frac{\norm{K^lf}_2^2 }{\norm{f}_2^2} \ge \left( \frac{\norm{Kf}_2^2}{\norm{f}_2^2} \right)^l.
 \end{equation}
For any finite set $A$ define
\[
 \lambda(A)= \sup_{\substack{\operatorname{supp}(f) \subseteq A,\\ f \not\equiv 0}} \frac{\norm{Kf}_2^2}{\norm{f}_2^2}.
\]
Then by \eqref{e-lb1} and Cauchy-Schwarz inequality
\begin{align}
\nonumber  \psi_K(n)= \norm{K^n}_{1\to 2}^2 &\ge \sup_A \sup_{\substack{\operatorname{supp}(f) \subseteq A, \\ \norm{f}_1=1}} \norm{f}_2^2\left( \frac{\norm{Kf}_2^2}{\norm{f}_2^2} \right)^n  \\
&\ge \sup_A \frac{\lambda(A)^n}{\mu(A)} \label{e-lb2}.
\end{align}
We write $\lambda(A)$ as
\begin{equation}
\lambda(A)= 1-( 1 - \lambda(A) )  =1 - \inf_{\substack{\operatorname{supp}(f) \subseteq A,\\ f \not\equiv 0}} \frac{\E_{K^2}(f,f)}{\norm{f}_2^2}\label{e-lb3}.
\end{equation}
To obtain a lower bound on $\lambda(A)$ it suffices to pick a test function $f$.
By Lemma \ref{l-dircomp}, Proposition \ref{p-sub}, there exists $C_1>0$ such that
\begin{equation} \label{e-lb4}
 \E_{K^2}(f,f) \le 2 \E_K(f,f) \le C_1 \E_{Q_\phi}(f,f) = C_1 c_\phi \sum_{n=1}^\infty \frac{1}{n \phi(n)} \E_{Q^{2\flr{n^\gamma}}}(f,f).
\end{equation}
By Lemma \ref{l-res}, \eqref{e-res} and \eqref{e-vc}, there exist constants $c_1 \in (0,1)$ and $C_2,C_3 > 1$ such that
\[
 R_Q(B(x,c_1 R), B(x,R)^c) \ge C_2^{-1} \frac{R^\gamma}{V(R)}
\]
for all $x \in \Gamma$ and for all $R \ge C_3$. Therefore for any $x \in \Gamma$ and for any $R>C_3$, there
exists $f \in \R^\Gamma$ satisfying $\operatorname{supp}(f) \subseteq B(x,R), f\restr_{B(x,c_1R)} \equiv 1$ and
\begin{equation} \label{e-lb5}
 \E_Q(f,f) \le \frac {2 C_2 V(R)}{ R^\gamma}.
\end{equation}
Since such a function has $\norm{f}_2^2 \ge V(x,c_1R)$, by \eqref{e-hom}, \eqref{e-vc} and  \eqref{e-lb5}, there exists $C_4>1$ such that the following holds: for any $x \in \Gamma$ and for any $R>C_3$, there
exists $f \in \R^\Gamma$ satisfying $\operatorname{supp}(f) \subseteq B(x,R)$ and
\begin{equation} \label{e-lb6}
 \frac{\E_Q(f,f)}{\norm{f}_2^2} \le C_4 R^{-\gamma}.
\end{equation}

Using Lemma \ref{l-dircomp} and the bound $\E_{Q^{2k}(f,f)} = \norm{f}_2^2 - \norm{Q^kf}_2^2 \le \norm{f}_2^2$, we obtain
\begin{equation} \label{e-lb7}
 \sum_{n=1}^\infty \frac{1}{n \phi(n)} \E_{Q^{2\flr{n^\gamma}}}(f,f)\le 2 \sum_{n=1}^{\flr{R}} \frac{n^{\gamma-1}}{\phi(n)} \E_Q(f,f) + \norm{f}_2^2 \sum_{n=\flr{R}+1}^\infty \frac{1}{n \phi(n)}
\end{equation}
for all $f \in \ell^{2}(\Gamma,\mu)$. For the second term above, we use \cite[Proposition 1.5.10]{BGT} to obtain $C_5 >0$ such that
\begin{equation} \label{e-lb8}
 \sum_{n=\flr{R}+1}^\infty \frac{1}{n \phi(n)} \le C_5 \frac{1}{\phi(R)}
\end{equation}
for all $R \ge 1$. By Potter's bound \cite[Theorem 1.5.6]{BGT} and continuity of $\phi$, there exists $C_6>0$ such that
\begin{equation} \label{e-lb9}
 \sum_{n=1}^{\flr{R}} \frac{n^{\gamma-1}}{\phi(n)} \le  C_6 \int_0^R \frac{s^{\gamma-1} \, ds}{\phi(s)}
\end{equation}
for all $R \ge 1$.

Combining \eqref{e-lb3}, \eqref{e-lb4}, \eqref{e-lb6}, \eqref{e-lb7}, \eqref{e-lb8}, \eqref{e-lb9} and using Lemma \ref{l-eta}(b), there exist constants $C_7>0$ and $R_0 >0$ such that
\[
 \lambda(B(x,R)) \ge 1 - \frac{C_7}{\eta(R)}
\]
for all $R > R_0$.
Combining \eqref{e-lb2}, \eqref{e-hom}, \eqref{e-inv} of Lemma \ref{l-eta}(c) along with the substitution $R= \zeta(n)$, there exists $N_1, C_8, c_1>0$ such that
\[
 \psi_K(n) \ge \frac{C_h^{-1}}{V(\zeta(n))} \left(1 - \frac{C_7}{\eta( \zeta(n))} \right)^n \ge   \frac{C_h^{-1}}{V(\zeta(n))} \left(1 - \frac{C_8}{n} \right)^n \ge \frac{c_1}{V(\zeta(n))}
\]
for all $n \in \N$ with $n \ge N_1$. The case $n \le N_1$ follows from \ref{e-ker}.
\end{proof}

\begin{proof}[Proof of Theorem \ref{t-main}]
 The upper bound and lower bound follows from Theorems \ref{t-ultra} and \ref{t-lb} respectively.
\end{proof}
\section{Stable subordination and the case \texorpdfstring{$\beta < \gamma$}{b-lt-g}}

In this section, we provide evidence to the conjecture in Remark \ref{r-main}(b) and (e). Let $(\Gamma,\mu)$ be a weighted graph
satisfying the volume doubling condition: there exists $C_D>0$ such that
\begin{equation} \label{e-gvd}
V_\mu(x,2r) \le C_D V_\mu(x,r)
\end{equation}
for all $x \in \Gamma$ and for all $r>0$.
By a slight abuse of notation, we denote $V_\mu$ by $V$ in this section.
We denote
Similar to \eqref{e-vc}, there is a volume comparison estimate
\begin{equation} \label{e-gvc}
\frac{V(x,r)}{V(x,s)} \le C_D \left( \frac{r}{s} \right)^\alpha
\end{equation}
for any $x \in \Gamma$ , for all $0 < s \le r$ and for all $\alpha \ge \log_2 C_D$.

As before let $P$ and $p_n$ denote the Markov operator corresponding to the natural random walk and the heat kernel respectively. We assume that
the heat kernel satisfies the following sub-Gaussian estimates.  There exist constants $c,C>0$ such that, for all $x,y \in \Gamma$
\begin{equation}\label{e-sgu}
 p_n(x,y) \le \frac{C}{V(x,n^{1/\gamma})} \exp \left[ - \left( \frac{d(x,y)^\gamma}{Cn} \right)^{\frac{1}{\gamma-1}} \right], \forall n \ge 1
\end{equation}
and
\begin{equation}\label{e-sgl}
(p_n+p_{n+1})(x,y) \ge \frac{c}{V(x,n^{1/\gamma})}  \exp \left[ - \left( \frac{d(x,y)^\gamma}{cn} \right)^{\frac{1}{\gamma-1}} \right], \forall n \ge 1 \vee d(x,y).
\end{equation}
Consider a random walk $X_n$ driven by the operator $Q$ defined in \eqref{e-defQ}.
We consider the continuous time Markov chain $Y_{\beta_0}(t) = X_{N(S_{\beta_0}(t))}$ where $N(t)$ and $S_{\beta_0}$ are independent Poisson process and $\beta_0$-stable subordinator for some $\beta_0 \in (0,1)$.
Let $k_{t,\beta_0}$ denote the kernel of $Y_{\beta_0}(t)$ with respect to the measure $\mu$. By definition of $k_{t,\beta_0}$, we have
\begin{equation} \label{e-defkt}
 k_{t,\beta_0}(x,y) = \sum_{i=0}^\infty  A_{\beta_0}(t,i) q_i(x,y)
\end{equation}
for all $t \ge 0$ and for all $x,y \in \Gamma$, where  $A_{\beta_0}(t,i):= \PP (N(S_{\beta_0}(t))= i) $.
Let $q_i$ denote the kernel of the iterated operator $Q^i$ with respect to the measure $\mu$ for $i \in \N$.
By the same proof as Lemma \ref{l-qk}, we get similar sub-Gaussian estimates for the more general volume doubling setup.  We assume that the kernel $q_n$ satisfies the following sub-Gaussian estimates:
There exist constants $c,C>0$ such that, for all $x,y \in \Gamma$
\begin{equation}\label{e-guq}
 q_n(x,y) \le \frac{C}{V(x,n^{1/\gamma})} \exp \left[ - \left( \frac{d(x,y)^\gamma}{Cn} \right)^{\frac{1}{\gamma-1}} \right], \forall n \ge 1
\end{equation}
and
\begin{equation}\label{e-qlq}
 q_n(x,y) \ge \frac{c}{V(x,n^{1/\gamma})}  \exp \left[ - \left( \frac{d(x,y)^\gamma}{cn} \right)^{\frac{1}{\gamma-1}} \right], \forall n \ge 1 \vee d(x,y).
\end{equation}
Using estimates on the stable subordinator $S_{\beta_0}$ and the estimates on the kernel $q_n$ similar to Lemma \ref{l-qk}, we show the following:
\begin{theorem} \label{t-evidence}
  Let $(\Gamma,\mu)$ be a weighted graph satisfying \eqref{e-gvd}  and
 its heat kernel $p_n$ satisfies the sub-Gaussian bounds \eqref{e-sgu} and \eqref{e-sgl} with escape time exponent $\gamma$. Let $k_{t,\beta_0}$ be the symmetric Markov kernel with respect to
 the measure $\mu$ defined by \eqref{e-defkt}. Then  for all $\beta_0 \in (0,1)$ there exists a constant $C>0$ such that
 \begin{equation} \label{e-evu}
 k_{n,\beta_0}(x,y)\le C \left( \frac{1}{ V(x,n^{1/\beta})} \wedge \frac{n}{ V(x,d(x,y))(1+d(x,y))^\beta} \right)
\end{equation}
 and
 \begin{equation} \label{e-evl}
  k_{n,\beta_0}(x,y) \ge C^{-1}\left( \frac{1}{ V(x,n^{1/\beta})} \wedge \frac{n}{ V(x,d(x,y))(1+d(x,y))^\beta} \right)
 \end{equation}
for all $x,y \in \Gamma$ and for all $n \in \N^*$, where $\beta =\beta_0 \gamma$.
\end{theorem}

We begin by recalling some known estimates for stable subordinator.
Let $f_{t, \beta_0}(u)$ be the density of the $\beta_0$-stable subordinator $S_{\beta_0}(t)$. We have the scaling relation
\[
 f_{t,\beta_0}(u)= t^{-1/\beta_0} f_{1,\beta_0}(t^{-1/\beta_0} u), \hspace{1cm} \beta_0 \in (0,1).
\]
By standard estimates on $f_{t,\beta_0}$ (see \cite[Section 3]{GS}) there exist constants $c_1,C_1>0$ such that
\begin{align}
 \label{e-fub}f_{t,\beta_0}(u) & \le C_1 t u^{-1-\beta_0}, \hspace{1cm} t,u >0, \\
  \label{e-fub1}f_{1,\beta_0}(u) & \le  C_1  u^{-\frac{2-\beta_0}{2 - 2 \beta_0}} e^{-c_1 u^{-\frac{\beta_0}{1-\beta_0}}}, \hspace{1cm} u \in (0,1), \\
 \label{e-flb} f_{t,\beta_0}(u) & \ge  c_1 t u^{-1-\beta_0}, \hspace{1cm} t>0, u > t^{1/\beta_0}.
\end{align}

Next, we estimate the quantity
\begin{equation} \label{e-Ab}
A_{\beta_0}(t,i)= \PP (N(S_{\beta_0}(t))= i) = \int_0^\infty f_{t,\beta_0}(u) \frac{e^{-u} u^i}{i!} \,du.
\end{equation}
By \eqref{e-fub} and Stirling asymptotics for Gamma function, there exists $C_2>0$
\begin{align}
 \nonumber  A_{\beta_0}(t,i)&\le  C_1 \int_0^\infty  t u^{-1-\beta_0}  \frac{e^{-u} u^i}{i!} \,du \\
 \label{e-ev1} &\le C_1 t i^{-1} \frac{\Gamma(i-\beta_0)}{\Gamma(i)} \le C_2 \frac{t}{i^{1+\beta_0}}
\end{align}
for all $t >0$ and for all $i \in \N^{*}$.
By Chebychev's inequality applied to Gamma distribution, we have
\begin{equation}
 \label{e-cheb}  \int_{\lambda/2}^\infty  \frac{e^{-u} u^{\lambda -1}}{\Gamma(\lambda)} \, du \ge \frac{1}{5}
\end{equation}
for all $\lambda \ge 5$.
Therefore, there exists $c_2>0$ such that
\begin{align}
 \label{e-ev2a}  A_{\beta_0}(t,i)&\ge  c_1 \int_{t^{1/\beta_0}}^\infty  t u^{-1-\beta_0}  \frac{e^{-u} u^i}{i!} \,du \\
 & \ge  c_1 \int_{(i - \beta_0)/2}^\infty  t u^{-1-\beta_0}  \frac{e^{-u} u^i}{i!} \,du \nonumber \\
 & \ge \frac{c_1}{5}   t \frac{\Gamma(i - \beta_0)}{i \Gamma(i)} \ge c_2 \frac{t}{i^{1+\beta_0}} \label{e-ev2}
\end{align}
for all $\beta_0 \in (0,1)$, for all $i \in \N^*$ and for all $t >0$ such that $i \ge \max\left(6, 4 t^{1/\beta_0}\right)$.
We used \eqref{e-flb} in the first line $i \ge \max\left(6, 4 t^{1/\beta_0}\right)$ in the second line and \eqref{e-cheb} and Stirling asymptotics for Gamma function in the last line.

We need the following estimate to prove the desired diagonal upper bound.
\begin{lemma} \label{l-poiss}
Under the doubing assumption \eqref{e-gvd}, there exists $C_1 >0$ such that
 \begin{equation}
 \label{e-poiss} \sum_{i=0}^ \infty \frac{\exp(-u) u^i}{i!} \frac{1}{ V(x,i^{1/\gamma})} \le \frac{C_1} { V(x, u^{1/\gamma})}
 \end{equation}
for all $x \in \Gamma$ and for all $u  \ge 0$.
\end{lemma}
\begin{proof}
Note that
\begin{align}
 \nonumber \sum_{i=0}^ \infty \frac{\exp(-u) u^i}{i!} \frac{1}{ V(x,i^{1/\gamma})} &\le \frac{1}{V(x,u^{1/\gamma})} \sum_{i=0}^{\flr{u}}\frac{\exp(-u) u^i}{i!} \frac{V(x,u^{1/\gamma})}{ V(x,i^{1/\gamma})} + \frac{1}{V(x,u^{1/\gamma})} \\
\nonumber & \le  \frac{C_2}{V(x,u^{1/\gamma})}  \sum_{i=0}^\infty  \frac{\exp(-u) u^i}{i!} \left( \frac{u}{i+1} \right)^{n_0}+ \frac{1}{V(x,u^{1/\gamma})} \\
\nonumber & \le  \frac{C_2 (n_0!)}{V(x,u^{1/\gamma})}  \sum_{i=0}^\infty  \frac{\exp(-u) u^{i+n_0}}{(i+n_0)!} + \frac{1}{V(x,u^{1/\gamma})} \\
& \le   \frac{C_3}{V(x,u^{1/\gamma})}
 \end{align}
where $n_0= \cil{(\log_2 C_D)/\gamma}$. We used \eqref{e-gvc} in the second line.
\end{proof}
\begin{proof}[Proof of Theorem \ref{t-evidence}]
We start by showing the off-diagonal lower bound for the case $d(x,y)^\gamma \ge 4 n^{1/\beta_0}$. By \eqref{e-defkt}, \eqref{e-ev2},\eqref{e-gvc} and \eqref{e-qlq}, we have
\begin{align}
 \nonumber k_{n,\beta_0}(x,y) &\ge c_1 \sum_{i=\cil{d(x,y)^\gamma}}^ {2 \cil{d(x,y)^\gamma}} \frac{n}{ (1+ i)^{1+\beta_0}} \frac{1}{V(x,d(x,y))} \\
 &\ge c_2  \frac{n}{ (1+ d(x,y))^{\beta}} \frac{1}{V(x,d(x,y))} \nonumber
\end{align}
for all $x,y \in \Gamma$, for all $n \in \N^*$ such that $d(x,y)^\gamma \ge 4 n^{1/\beta_0}$.
Next, we show the near-diagonal lower bound for the case $d(x,y)^\gamma \le 4 n^{1/\beta_0}$. By \eqref{e-defkt}, \eqref{e-ev2},\eqref{e-gvc} and \eqref{e-qlq}, we have
\[
 k_{n,\beta_0}(x,y) \ge c_3 \sum_{i=\cil{4 n^{1/\beta_0}}}^ {\cil{8 n^{1/\beta_0}}} \frac{n}{ (1+ i)^{1+\beta_0}} \frac{1}{V(x,n^{1/\beta})} \ge   \frac{c_4}{V(x,n^{1/\beta})} \nonumber
\]
for $x,y \in \Gamma$ and for all $n \in \N^*$ such that $d(x,y)^\gamma \le 4 n^{1/\beta_0}$.

We prove the diagonal upper bound below. We use \eqref{e-defkt}, \eqref{e-Ab} and Fubini's theorem to obtain
\begin{equation} \label{e-ev4}
 k_{n,\beta_0}(x,y) = \int_0^\infty f_{n,\beta_0}(u) \sum_{i=0}^\infty \frac{e^{-u} u^i}{i!} q_i(x,y) \, du.	
\end{equation}
Combining \eqref{e-guq} ,\eqref{e-ev4} and Lemma \ref{l-poiss}, there exists $C_2,C_3,C_4>0$ such that
\begin{align}
\nonumber  k_{n,\beta_0}(x,y) &\le C_1 \int_0^\infty f_{n,\beta_0}(u) \sum_{i=0}^\infty \frac{e^{-u} u^i}{i!}\frac{1}{V(x,i^{1/\gamma})} \, du  \\
\nonumber & \le  C_2   \int_0^\infty f_{n,\beta_0}(u) \frac{1}{V(x,u^{1/\gamma})} \, du \\
\nonumber & =  C_2 \int_0^\infty f_{1,\beta_0}(s)\frac{1}{V(x, n^{1/\beta} s^{1/\gamma})} \, ds \\
\nonumber & \le  \frac{C_2} {V(x, n^{1/\beta})} + \frac{C_3} {V(x, n^{1/\beta})}\int_0^1  s^{- \frac{2 -\beta_0}{2 - 2 \beta_0}} e^{-c_1 s^{-\frac{\beta_0}{1 -\beta_0} }} \frac{1}{ s^{(\log_2 C_D)/\gamma}} \, ds  \\
& \le  \frac{C_4} {V(x, n^{1/\beta})}.
\end{align}

Next, we show the off-diagonal upper bound in \eqref{e-evu}.
Combining \eqref{e-defkt}, \eqref{e-ev1}, \eqref{e-guq}, there exists $C_5,C_6,C_7 >0$ such that
\begin{align}
 \nonumber \lefteqn{ k_{n,\beta_0}(x,y) }\\
\nonumber & \le   C_5  n  \sum_{i=1}^\infty  (1+ i)^{-1 -\beta_0}  \frac{C}{V(x,i^{1/\gamma})} \exp \left[ - \left( \frac{d(x,y)^\gamma}{Ci} \right)^{\frac{1}{\gamma-1}} \right] \\
\nonumber & \le   \frac{  C_6 n}{(1+d(x,y))^\beta V(x,d(x,y))}\left( (1 + d(x,y))^\beta\sum_{i=1+ \flr{d(x,y)^\gamma}}^\infty (1+i)^{-1 -\beta_0}  \right.  \\
\nonumber & \hspace{4mm} +\left. d(x,y)^{-\gamma} \sum_{i=1}^{\flr{d(x,y)^\gamma}}    \left( \frac{d(x,y)^\gamma}{i} \right)^{1 +\beta_0 + (\alpha/\gamma)} \exp \left[   - \left( \frac{d(x,y)^\gamma}{Ci} \right)^{\frac{1}{\gamma-1}}  \right] \right)\\
& \le  \frac {C_7 n }{(1+d(x,y))^\beta V(x,d(x,y))} \label{e-odub}
\end{align}
for all $x,y \in \Gamma$ with $x \neq y$ and for all $n \in \N$.
\end{proof}



\begin{thebibliography}{99}

\bibitem{Bar} M. T. Barlow, \textit{Which values of the volume growth and escape time exponent are possible for a graph?},  Rev. Mat. Iberoamericana 20 (2004), no. 1, 1\textendash31. \mr{2076770}

\bibitem{BBK} M. T. Barlow, R. F. Bass, T. Kumagai, \textit{Parabolic Harnack inequality and heat kernel estimates for random walks with long range jumps},  Math. Z. 261 (2009), no. 2, 297\textendash320.
\mr{2457301}.

\bibitem{BCK} M. T. Barlow, T. Coulhon, T. Kumagai, \textit{Characterization of sub-Gaussian heat kernel estimates on strongly recurrent graphs},   Comm. Pure Appl. Math. 58 (2005), no. 12, 1642\textendash1677. \mr{2177164}

\bibitem{BGK}
M. T. Barlow,  A. Grigor'yan, T. Kumagai,  \textit{Heat kernel upper bounds for jump processes and the first exit time},  J. Reine Angew. Math. 626 (2009), 135\textendash157. \mr{2492992}

\bibitem{BL}
R. F. Bass, D. A. Levin, \textit{Transition probabilities for symmetric jump processes},  Trans. Amer. Math. Soc. 354 (2002), no. 7, 2933\textendash2953. \mr{1895210}

\bibitem{BS1}
A. Bendikov, L. Saloff-Coste,  \textit{Random walks on groups and discrete subordination},  Math. Nachr. 285 (2012), no. 5-6, 580\textendash605.  \mr{2902834}

\bibitem{BS2}A. Bendikov, L. Saloff-Coste,  \textit{Random walks driven by low moment measures},  Ann. Probab. 40 (2012), no. 6, 2539\textendash2588. \mr{3050511}


\bibitem{BGT}
N. H. Bingham, C. M. Goldie, and J. L. Teugels, \textit{Regular variation}, Encyclopedia of Mathematics and its Applications, vol. 27, Cambridge University
Press, Cambridge, 1987. \mr{898871}


\bibitem{Cou}
T. Coulhon, \textit{Ultracontractivity and Nash Type Inequalities}  J. Funct. Anal. 141 (1996), no. 2, 510\textendash539. \mr{1418518}


\bibitem{CK1} Z.-Q. Chen, T. Kumagai, \textit{Heat kernel estimates for stable-like processes on $d$-sets}, Stochastic Process. Appl. 108 (2003), no. 1, 27\textendash62. \mr{2008600}

\bibitem{CK2} Z.-Q. Chen, T. Kumagai, \textit{Heat kernel estimates for jump processes of mixed types on metric measure spaces},  Probab. Theory Related Fields 140 (2008), no. 1-2, 277\textendash317. \mr{2357678}

\bibitem{CKS} E. A. Carlen, S. Kusuoka, D. W. Stroock, \textit{Upper bounds for symmetric Markov transition functions},
Ann. Inst. H. Poincar\'{e} Probab. Statist. 23 (1987), no. 2, suppl., 245\textendash287. \mr{0898496}

\bibitem{CG} T. Coulhon, A. Grigor'yan, \textit{On-diagonal lower bounds for heat kernels and Markov chains}, Duke Math. J. 89 (1997), no. 1, 133\textendash199. \mr{1458975}

\bibitem{Del}  T. Delmotte, \textit{Parabolic Harnack inequality and estimates of Markov chains on graphs},  Rev. Mat. Iberoamericana 15 (1999), no. 1, 181\textendash232. \mr{1681641}

\bibitem{MS1} M. Murugan, L. Saloff-Coste, \textit{Transition probability estimates for long range random walks}, \arxiv{1411.2706}

\bibitem{PS} Ch. Pittet, L.  Saloff-Coste, \textit{On the stability of the behavior of random walks on groups},  J. Geom. Anal. 10 (2000), no. 4, 713\textendash737. \mr{1817783}

\bibitem{PSu} Ch. Pittet, L. Saloff-Coste, \textit{A survey on the relationships between volume growth, isoperimetry and the behavior of simple random walk on Cayley graphs, with examples} (in preparation)

\bibitem{GS} P. Graczyk, A. St\'{o}s, \textit{Transition density estimates for stable processes on symmetric spaces},  Pacific J. Math. 217 (2004), no. 1, 87\textendash100. \mr{2105767}

\bibitem{GHL} A. Grigor'yan,  J. Hu, K.-S. Lau,  \textit{Heat kernels on metric spaces with doubling measure}, Fractal geometry and stochastics IV, 3\textendash44, Progr. Probab., 61, Birkh\"{a}user Verlag, Basel, 2009. \mr{2762672}

\bibitem{GT0} A. Grigor'yan, A. Telcs, \textit{Sub-Gaussian estimates of heat kernels on infinite graphs}, Duke Math. J. 109 (2001), no. 3, 451\textendash510. \mr{1853353}

\bibitem{GT} A. Grigor'yan, A. Telcs, \textit{Harnack inequalities and sub-Gaussian estimates for random walks}, Math. Ann. 324 ,  551-556 (2002). \mr{1938457}

\bibitem{Kum} T. Kumagai, \textit{Anomalous random walks and diffusions: From fractals to random media}, Proceedings of 2014 ICM, Seoul.

\bibitem{S-C} L. Saloff-Coste, \textit{A note on Poincar\'{e}, Sobolev, and Harnack inequalities}, Internat. Math. Res. Notices 1992, no. 2, 27\textendash38. \mr{1150597}

\bibitem{SZ} L. Saloff-Coste, T. Zheng, \textit{ On some random walks driven by spread-out measures}, \arxiv{1309.6296} .
\end{thebibliography}


\subsection*{Acknowledgements}  We thank the referee for helpful comments and careful reading of the manuscript.

\end{document}